\documentclass[10pt]{amsart}
\usepackage{amsmath}
\usepackage{amssymb}
\usepackage{amsthm}
\usepackage{mathrsfs}
\usepackage{comment}
\usepackage{hyperref}

\usepackage[all,cmtip]{xy}\usepackage{xcolor}
\usepackage{enumerate}
\usepackage{bm}
\usepackage{dsfont}
\usepackage{mathtools}
\usepackage[cal=euler]{mathalfa}
\usepackage[top=1in, bottom=1.25in, left=1.25in, right=1.25in]{geometry}
\usepackage{parskip}

\hypersetup{colorlinks=true,linkcolor=magenta,citecolor=blue}

\DeclareMathAlphabet{\mathpzc}{OT1}{pzc}{m}{it}

\theoremstyle{plain}
\newtheorem{theorem}{Theorem}[section]
\newtheorem*{theorem*}{Theorem}

\newtheorem{lemma}[theorem]{Lemma}
\newtheorem*{claim*}{Claim}
\newtheorem{proposition}[theorem]{Proposition}

\newtheorem{corollary}[theorem]{Corollary}

\theoremstyle{definition}

\newtheorem{remark}[theorem]{Remark}

\newcommand{\can}{\overline{\phantom{x}}}

\newcommand{\ignore}[1]{}

\begin{document}
\setlength{\parindent}{0pt}



\author{S. Pumpluen}
\email{susanne.pumpluen@nottingham.ac.uk}
\address{School of Mathematical Sciences\\
University of Nottingham\\
University Park\\
Nottingham NG7 2RD\\
United Kingdom
}



\title{Colour algebras over rings}

\begin{abstract}  
Colour algebras over fields of odd characteristic are well-known noncommutative Jordan algebras.
We define colour algebras more generally over a unital commutative associative ring with $\frac{1}{2}\in R$, and show that colour algebras can be constructed canonically by employing nondegenerate ternary hermitian forms with trivial determinant. We investigate their structure, automorphism group and derivations.
As over fields, colour algebras over $R$ are closely related to octonion algebras over $R$.
\end{abstract}

\maketitle

\section*{Introduction}

Let $F$ be a field of characteristic not 2.
The \emph{split colour algebra} ${\rm Col}(F)$ is defined as the seven-dimensional algebra with basis $1,u_i,v_i$ for $1\leq i\leq 3$ such that $u_iu_j=\varepsilon_{i,j,k}v_k$, $v_iv_j=\varepsilon_{i,j,k}u_k$, and $u_iv_j=v_iu_j=\delta_{ij}1$, where $\varepsilon_{i,j,k}$ is the totally skew-symmetric tensor with $\varepsilon_{1,2,3}=1$. In the original definition, $F=\mathbb{C}$  and ${\rm Col}(\mathbb{C})$ is employed to describe the colour symmetries of the quark model \cite{DKD}.
A \emph{colour algebra} $A$ over $F$ is a form of the split colour algebra ${\rm Col}(F)$. Colour algebras over fields  were investigated  in \cite{E1, EM, E, Sch1, Sch2}.

Let $R$ be a unital commutative associative ring where $2\in R^\times$ is invertible.  A unital algebra $A$ over $R$ which is  finitely generated projective  of constant rank as an $R$-module and has full support, is called a \emph{colour algebra} if $A(P)=A\otimes k(P)$ is a colour algebra over the field $k(P)$ for all $P \in {\rm Spec}(R)$.
Let $T$ be a projective $R$-module of constant rank 3 such that $\bigwedge^3 (T)\cong R $ via some isomorphism $\alpha$.
The split colour algebras ${\rm Col}( T,\alpha)$ over $R$ with underlying $R$-module structure $R\oplus T\oplus \check{T}$ generalize the ``classical'' split colour algebra ${\rm Col}(R)$ which is defined on a free $R$-module, and are closely related to Zorn algebras over rings. They
  form an important first class of colour algebras, and are introduced in Section \ref{sec:split} after the basic definitions are collected in Section \ref{sec:1}. The construction of ${\rm Col}( T,\alpha)$ is functorial in the parameters involved (Proposition \ref{prop:diagiso}).
 In Section \ref{sec:hermitian}, we construct colour algebras employing nondegenerate ternary hermitian forms with trivial determinant and in Section \ref{sec:iso} we generalise results on isomorphisms, automorphisms and derivations from \cite{EM} to colour algebras over rings.

For all $l,m\in \mathbb{Z}$, we then use split colour algebras over the $n$-dimensional projective space $\mathbb{P}_R^n$ over $R$ in Section \ref{sec:Jordan} to construct a noncommutative Jordan subalgebra of the split colour algebra $ {\rm Col}(R[t_0,\dots,t_n])$ that is  a free $R$-module of rank $1+\binom{l+n}{n}+\binom{m+n}{n}+\binom{(l+m)+n}{n}$.
Here,
 $R[t_0,\dots,t_n]$ is the polynomial ring in $n+1$ variables over $R$. When $R$ is a field, these noncommutative Jordan algebras have highly degenerate norms and therefore large radicals, analogously as discussed in a similar construction but employing Zorn algebras in \cite[3.8]{P2}.

Colour algebras over fields are unital central simple algebras, and appear as one of the two non-trivial cases in the classification of finite-dimensional central simple noncommutative Jordan algebras over a field $F$ of odd characteristic \cite{Sch3}:  Every colour algebra is flexible and quadratic and therefore a noncommutative Jordan algebra.

In physics, the colour symmetry of the Gell-Man quark model can be described as the multiplication of a colour algebra. For each quark, there is an antiquark which has the opposite properties of the quark. Quarks and antiquarks are used to form particles called hadrons. Each quark comes in three varieties and colour was used to describe the interactions of quarks. For an accessible explanation of the construction of the resulting colour algebra, see \cite{W}.

It is well-known that the structure of octonion algebras over rings has a much richer flavour than the theory of octonion algebras over fields.
Because of their intricate connections with colour algebras, this is reflected in the structure of colour algebras over rings, and in the vector products associated with both algebras.

\section{Preliminaries}\label{sec:1}

Let $R$ be a unital commutative associative ring and $2$ be an invertible element in $R$.
  For $P \in {\rm Spec}(R)$ let $R_P$ be the localization of $R$ at $P$  and $m_P$ the maximal ideal of $R_P$. We denote the corresponding residue class field by $k(P)=R_P/m_P$. For an $R$-module $M$ the localization of
$M$ at $P$ is denoted by $M_P$. An $R$-module $M$  has {\it full support}  if $M_P\not=0$ for all $P\in {\rm Spec}(R)$.

 All nonassociative $R$-algebras $A$ considered in this paper are finitely generated projective of constant rank as an $R$-module and have full support.

 A unital algebra $A$ over $R$ is called a \emph{colour algebra} if $A(P)=A\otimes k(P)$ is a colour algebra over the field $k(P)$ for all $P \in {\rm Spec}(R)$.

 Any anti-automorphism $\sigma:A\rightarrow A$ of period two is called an {\it involution} on $A$.
An involution $\sigma$ is called {\it scalar} if  $\sigma(x)x\in R1$ for all $x\in A$.
For every scalar involution $\sigma$,  the {\it norm} $n_A:A\rightarrow R$, $n_A(x)=\sigma(x)x$ (resp. the {\it trace} $t_A:A\rightarrow R$, $t_A(x)=\sigma(x)+x$) is a quadratic (resp. an $R$-linear) form.

 A unital algebra $A$ is called {\it quadratic}, if there exists a quadratic
form $n \colon A \rightarrow R$ such that $n(1_A) = 1$ and $x^2 - n(1_A, x)x + n(x)
1_A = 0$ for all $x \in A$, where $n(x,y)$ denotes the
induced symmetric bilinear form  $n(x, y) = n (x+y) - n(x) -n(y)$. The form $n$ is uniquely determined
and called the {\it norm} of the quadratic algebra $A$ \cite{P2}.  If an algebra $A$ has a scalar involution then  $A$ is a quadratic algebra \cite{M1}.

If $A$ is a quadratic algebra over $R$ then $A=R1\oplus A_0$ with $A_0=\{ u\in A\,|\, t(u)=0\}$. Define $\times: A_0\times A_0\rightarrow A_0$,
$u\times v=pr(uv)$ with $pr:A \rightarrow A_0$ the canonical projection map. Then $(A_0,\times)$ is an anti-commutative algebra over $R$ with $uv=-\frac{1}{2}n(u,v)+u\times v$.
The algebra $(A_0,\times)$ is called the \emph{vector algebra} of $A$.

A unital algebra $C$ over $R$ is  a
{\it composition algebra} if there exists a quadratic form $n \colon C \rightarrow R$
whose induced symmetric bilinear form $n(x, y)=n(x+y)-n(x)-n(y)$ is  nondegenerate, i.e., $n$ determines an $R$-module isomorphism $C
\stackrel{\sim}{\longrightarrow} \check{C} = {\rm Hom}_R (C, R)$, and  satisfies
 $n(xy) = n(x) n(y)$ for all $x, y \in C$.
 Composition algebras are quadratic algebras. A nondegenerate quadratic form $n$ on  $C$
 which satisfies $n(xy)=n(x)n(y)$ for all $x,y \in C$ is its norm as a quadratic algebra and thus uniquely determined up to isometry \cite{P2}. It
is called the {\it norm} of  $C$.  Composition algebras only
exist in ranks 1, 2, 4 or 8.  Composition algebras of rank~2 are exactly the
quadratic \'etale algebras over $R$ (these are sometimes called \emph{tori} in the literature).
The composition algebras of rank four are called \emph{quaternion algebras}, they are  Azumaya algebras of rank four.
Composition algebras of rank eight are called {\it octonion algebras}.  A composition algebra $C$ over $R$ is called \emph{split} if $C$ contains an isomorphic copy of the \emph{split torus} $R\times R$ with isotropic norm $n((a,b))=ab$ as a composition subalgebra.

 Every composition algebra $C$ has a {\it canonical involution} $\can$ given by
$\overline{x} = t(x)1_C - x$, where $t \colon C \rightarrow R$, $t(x) = n (1_C, x)$,
is the {\it trace} of $C$. This involution is scalar.  We know that $n(x)=x\bar x$ for all $x\in C$.

 When $2\in R^\times$, then all the octonion algebras over $R$ with a quadratic \'etale subalgebra can be constructed employing a hermitian space of rank three with trivial determinant \cite{T}.

Let $\times:R^3\times R^3\rightarrow R^3$ denote the classical vector product on the three-dimensional column space $R^3$. The algebra
\[
{\rm Zor}(R)=
\left [\begin {array}{cc}
 R & R^3\\
R^3  &  R \\
\end {array}\right ],
\]
\[
\left [\begin {array}{cc}
 a & u\\
u'  &  a' \\
\end {array}\right ]
\left [\begin {array}{cc}
 b & v\\
v'  &  b' \\
\end {array}\right ]
=
\left [\begin {array}{cc}
 ab+ ^tuv' & av+b'u-u'\times v'\\
bu' +a'v'+u\times v & ^tu'v+ a'b' \\
\end {array}\right ],
\]
is a split octonion algebra over $R$ with norm
\[
{\rm det}
\left [\begin {array}{cc}
 a & u\\
u'  &  a' \\
\end {array}\right ]=
aa'-^tuu'
\]
and is called \emph{Zorn's algebra of vector matrices}.  If $R$ is a field or more generally, a principal ideal domain or a Dedekind domain, then ${\rm Zor}(R)$ is, up to isomorphism, the only split octonion algebra over $R$.

\section{Colour algebras over rings}\label{sec:2}

\subsection{Split colour algebras over rings}\label{sec:split}
Let $T$ be a projective $R$-module of constant rank 3 such that ${\rm det}\,T=\bigwedge^3 (T)\cong R $.
Let $\langle \,,\,\rangle:T\times\check{T}\rightarrow R,$ $ \langle u,\check{v}\rangle=\check{v}(u)$
be the canonical pairing between $T$ and its  dual module $\check{T}={\rm Hom}_R(T,R)$.
Every isomorphism $\alpha:\bigwedge^3 (T)\rightarrow R $ induces a bilinear map
$\times:T\times T\rightarrow\check{T}$ via
$$(u,v)\mapsto u\times v=\alpha(u \wedge v \wedge -).$$
Such a map is called a {\it vector product} on $T$, since locally it is the ordinary vector product. Moreover,
$\alpha$ also determines an isomorphism $\beta:{\rm det}\check{T}\rightarrow  R$
which satisfies
$$\alpha(u_1 \wedge u_2 \wedge u_3)\beta(\check{u}_1 \wedge \check{u}_2 \wedge \check{u}_3)={\rm det}(\langle u_i,\check{u}_j\rangle)$$
for all $u_i\in T$, $\check{u}_j\in\check{T}$, $1\leq i,j\leq 3$. Therefore we analogously obtain a vector product
$\check{T}\times\check{T} \longrightarrow T$
$$(\check{u},\check{v})\mapsto \check{u}\times \check{v}=\beta(\check{u} \wedge \check{v} \wedge -)$$
 on $\check{T}$, employing $\beta$ instead of $\alpha$.

 Consider the finitely generated projective $R$-module of constant rank seven
 $${\rm Col}(T,\alpha)= R\left [\begin {array}{cc}
 1 & 0\\
0  &  1 \\
\end {array}\right ]\oplus T\oplus\check{T}= \{ \left [\begin {array}{cc}
 a & 0\\
0  &  a \\
\end {array}\right ]\,|\, a\in R \} \oplus
\{ \left [\begin {array}{cc}
 0 & u\\
\check{u}  &  0 \\
\end {array}\right ] \,|\,  u\in T, \check{u}\in \check{T}
\}$$
 $$= \{ \left [\begin {array}{cc}
 a & u\\
\check{u}  &  a \\
\end {array}\right ] \,|\, a\in R, u\in T, \check{u}\in \check{T}
\}$$
together with the multiplication given by
$$
\left [\begin {array}{cc}
 a & u\\
\check{u}  &  a \\
\end {array}\right ]
\left [\begin {array}{cc}
 b & v\\
\check{v}  &  b \\
\end {array}\right ]
=
\left [\begin {array}{cc}
  ab+\langle u,\check{v}\rangle + \langle v,\check{u}\rangle & av+bu-\check{u}\times \check{v}\\
b\check{u} +a\check{v}+u\times v &  ab+\langle u,\check{v}\rangle + \langle v,\check{u}\rangle \\
\end {array}\right ].
$$
The algebra ${\rm Col}(T,\alpha)$ is a colour algebra and is called  a \emph{split colour algebra} over $R$.
A split colour algebra is a quadratic algebra over $R$ with norm and trace
\[
n(
\left [\begin {array}{cc}
 a & u\\
\check{u}  &  a \\
\end {array}\right ])=
a^2-\langle u,\check{u}\rangle, \quad
{\rm tr}(
\left [\begin {array}{cc}
 a & u\\
\check{u}  &  a \\
\end {array}\right ])=
2a.
\]
The split colour algebra ${\rm Col}(T,\alpha)$ is flexible: $[x,y,x]=0$ for all $x,y\in {\rm Col}(R\times R,T,\alpha)$, and thus ${\rm Col}(T,\alpha)$ is power-associative. It is  a noncommutative Jordan algebra over $R$ as it satisfies the \emph{Jordan identity} $(x,y,x^2)=0$ for all $x,y\in {\rm Col}(T,\alpha)$.

Projecting the multiplication of ${\rm Col}(T,\alpha)$ onto the submodule $T\oplus \check{T} $ yields the 6-dimensional \emph{split vector colour algebra} $W(T,\alpha)=W(R\times R,T\oplus \check{T},h,\alpha)=T\oplus \check{T}$ over $R$
 with multiplication
 $$(u,\check{u})( v,\check{v})=(-\check{u}\times \check{v}, u\times v).$$
 The multiplication of $W(T,\alpha)$ is anti-commutative and satisfies the identity
 $$(((u,\check{u}) ( v,\check{v})) ( v,\check{v}))( v,\check{v}))=
 n(( v,\check{v}))(u,\check{u}) ( v,\check{v})$$
 with $n(( v,\check{v}))= -\langle v,\check{v}\rangle,$
 analogously as in \cite{E1} (and adjusted by the factor $\frac{1}{2}$ because of our different definition of the norm).

Split colour algebras are closely connected to  split octonion algebras, as  the split vector colour algebra $W(T,\alpha)$ of the split colour algebra ${\rm Col}(T,\alpha)$ is the same as the algebra we obtain when projecting the multiplication of a the Zorn algebra ${\rm Zor}(T,\alpha)$ to $T\oplus \check{T}$. Here,
\[
{\rm Zor}(T,\alpha)=
\left [\begin {array}{cc}
 R & T\\
\check{T}  &  R \\
\end {array}\right ]
\]
with the multiplication
\[
\left [\begin {array}{cc}
 a & u\\
\check{u}  &  a' \\
\end {array}\right ]
\left [\begin {array}{cc}
 b & v\\
\check{v}  &  b' \\
\end {array}\right ]
=
\left [\begin {array}{cc}
 ab+\langle u,\check{v}\rangle & av+b'u-\check{u}\times \check{v}\\
b\check{u} +a'\check{v}+u\times v & \langle v,\check{u}\rangle+ a'b' \\
\end {array}\right ]
\]
is a split octonion algebra over $R$  with norm
\[
{\rm det}
\left [\begin {array}{cc}
 a & u\\
\check{u}  &  a' \\
\end {array}\right ]=
aa'-\langle u,\check{u}\rangle .
\]
The octonion algebra ${\rm Zor}(T,\alpha)$ is  called a {\it Zorn algebra}. Every split octonion algebra over $R$ is isomorphic to such a Zorn algebra. Locally, ${\rm Zor}(T,\alpha)$ looks like ${\rm Zor}(R)$
\cite[3.3, 3.4, 3.5]{P2}.

It is now straightforward to see that every split colour algebra ${\rm Col}(T,\alpha)$ can be constructed from the vector part $(T,\alpha)$ of a Zorn algebra ${\rm Zor}(T,\alpha)$ over $R$.

When $T=R^3$ and $\times_\alpha=\times:R^3\times R^3\rightarrow R^3$ is the classical vector product on $R^3$, then the split colour algebra  over $R$  can be written as
\[
{\rm Col}(R)=
\left [\begin {array}{cc}
 R & R^3\\
R^3  &  R \\
\end {array}\right ],
\]
\[
\left [\begin {array}{cc}
 a & u\\
u'  &  a \\
\end {array}\right ]
\left [\begin {array}{cc}
 b & v\\
v'  &  b \\
\end {array}\right ]
=
\left [\begin {array}{cc}
 ab+ ^tuv'+^tu'v & av+bu-u'\times v'\\
bu' +av'+u\times v & ab+ ^tuv'+^tu'v \\
\end {array}\right ],
\]
with norm
\[
{\rm det}
\left [\begin {array}{cc}
 a & u\\
u'  &  a \\
\end {array}\right ]=
a^2-^tu u'.
\]
Again, the split colour algebra ${\rm Col}(R)$ can be constructed from the vector part of ${\rm Zor}(R)$. Locally, ${\rm Col}(T,\alpha)$ looks like ${\rm Col}(R)$.

Suppose that $T,T'$ are two  projective $R$-modules of constant rank 3 such that ${\rm det}\,T\cong {\rm det}\,T'\cong R $, and such that
 $\alpha:\bigwedge^3 (T)\rightarrow R $, $\alpha':\bigwedge^3 (T')\rightarrow R $ are two isomorphisms.
An $R$-linear map $\varphi: T\rightarrow T' $ such that $\alpha'\circ \det (\varphi)=\alpha$ is called a \emph{morphism} between $(T,\alpha)$ and $(T',\alpha') $ and denoted by $ (T,\alpha)\rightarrow (T',\alpha') $.
   If $\varphi: (T,\alpha)\rightarrow (T',\alpha') $
 is such a morphism then $\varphi$ is bijective, and induces ``diagonal'' isomorphisms between the corresponding split colour algebras:

 \begin{proposition}\label{prop:diagiso}
 If  $ \varphi:(T,\alpha)\rightarrow (T',\alpha') $
 is a morphism,  then
 $${\rm Col}( T,\alpha)\cong {\rm Col}( T',\alpha')$$
$$
\left [\begin {array}{cc}
 a & u\\
\check{u}  &  a \\
\end {array}\right ] \mapsto
\left [\begin {array}{cc}
 a & \varphi(u)\\
\check{ \varphi}^{-1}(\check{u})  &  a' \\
\end {array}\right ],
$$
and
$${\rm Col}( T,\alpha)\cong {\rm Col}(\check{T},\beta),$$
 $$
\left [\begin {array}{cc}
 a & u\\
\check{u}  &  a \\
\end {array}\right ] \mapsto
\left [\begin {array}{cc}
 a & u\\
\check{u}  &  a \\
\end {array}\right ]^{*t}
=
\left [\begin {array}{cc}
 a & - \check{u}\\
-u &  a \\
\end {array}\right ]^*.
$$
are isomorphic split colour algebras.
\end{proposition}

Hence the construction of ${\rm Col}( T,\alpha)$ is functorial in the parameters involved, analogously as the one of ${\rm Zor}(T,\alpha)$:  The two maps defined in Proposition \ref{prop:diagiso} also yield the  isomorphisms
${\rm Zor}(T,\alpha)\cong {\rm Zor}(T',\alpha')$ and $ {\rm Zor}(T,\alpha)\cong {\rm Zor}(\check{T},\beta)$
 cf. \cite[3.4]{P2}.

In other words, $(T,\alpha)$  determines the  algebras ${\rm Zor}(T,\alpha)$ and ${\rm Col}(T,\alpha)$ up to isomorphism, as well as
 the associated 6-dimensional split vector colour algebra $W(T,\alpha)$.

\begin{corollary}
If there is an $R$-linear map $\varphi:T\to T$  such that ${\rm det}(\varphi)=id$
 then
\[
\left [\begin {array}{cc}
 a & u\\
\check{u}  &  a \\
\end {array}\right ]\mapsto
\left [\begin {array}{cc}
 a & \varphi(u)\\
\check{\varphi}^{-1}(\check{u})  &  a \\
\end {array}\right ]
\]
lies in ${\rm Aut}_R({\rm Col}( T,\alpha))$.
 In particular, if there is $\mu\in R^\times$ such that $\mu^3=1$ then
\[
\left [\begin {array}{cc}
 a & u\\
\check{u}  &  a \\
\end {array}\right ]\to
\left [\begin {array}{cc}
 a & \mu u\\
\mu^{ -1}(\check{u})  &  a \\
\end {array}\right ] .
\]
lies in ${\rm Aut}_R({\rm Col}( T,\alpha))$,
i.e. if $R$ contains a primitive third root of unity then there exist two nontrivial automorphisms induced by  the two primitive third roots of unity in $R$.
\end{corollary}

\subsection{A construction of colour algebras employing hermitian forms}\label{sec:hermitian}

The construction of split colour algebras is part of a bigger picture; when $2\in R^\times$, every octonion algebra which contains a quadratic  \'etale  subalgebra
can be constructed employing a nondegenerate hermitian form of rank three with trivial determinant (Petersson and Racine \cite[3.8]{PR}, or Thakur \cite{T}). We now employ the elements of this construction for colour algebras.

 Let $S$ be a quadratic  \'etale $R$-algebra with canonical involution $\can$ and let $P$ be a finitely generated projective $R$-module of constant rank. A $\can$-hermitian form $h \colon P \times P \to S$ is a biadditive map
$h \colon P \times P \to D$ with $h(ws, w't) = \bar s h (w, w') t$ and
$h(w, w') = \overline{ h(w', w)}$ for all $s, t \in S$, $w, w' \in P$, and where
$P \to\; \overline{P}^{\vee}$, $w \mapsto h (w,\cdot)$ is an isomorphism of
$S$-modules.

Let $P$ have rank three and let $(P,h)$ be a nondegenerate $\can$-hermitian space such that $\bigwedge^3 (P,h)
\cong \langle 1\rangle $, where $\langle 1\rangle$ is the $\can$-hermitian form $h_0(a,b)=a\bar b$ on $S$. Then  $$n:P\times P\rightarrow R, \quad n(u,v)=h(u,v)+\overline{h(u,v)}$$
 is a nondegenerate symmmetric $R$-bilinear form.

 Choose an isomorphism $\alpha: \bigwedge^3 (P,h)\rightarrow \langle 1\rangle $.
Define a cross product $\times_{\alpha}:P\times P\rightarrow P$ via
$$h(u,v\times_\alpha w)=\alpha(u\wedge v \wedge w)$$
for all $u,v,w\in P$ as in \cite[p.~5122]{T} or \cite[3.8]{PR}. Note that this cross product
$\times_{\alpha}$ depends only on $T,P,h$ and $\alpha$.
The $R$-module
$C={\rm Col}(S,P,h,\alpha)=R\oplus P$ becomes a colour algebra  under the multiplication
$$(a,u)(b,v)=(ab-\frac{1}{2}(h(u,v)+\overline{h(u,v)}), va+u b+u\times_{\alpha}v)$$
for all $u,v\in P$ and $a,b\in R$.
The colour algebra ${\rm Col}(S,P,h,\alpha)$ is flexible and quadratic  with norm
$$n((a,u))=a^2+h(u,u),$$
and
$$n((a,u),(b,v))=n_S(a,b)+(h(u,v)+\overline{h(u,v)})=(a,u)\sigma(b,v)+(b,v)\sigma(a,u).$$

In  ${\rm Col}(S,P,h,\alpha)$ we have
 $$((u\times v)\times v)\times v=n(v) u\times v$$
for all $u,v\in P$ \cite[Lemma 2]{Pu2}. It is straightforward to check that
$$a(u\times_\alpha v)=(\bar a u)\times_\alpha v= u \times_\alpha \bar a v, $$
$$h(u,v\times_\alpha  w)=\overline{h(u \times_\alpha v,w)}=h(w,u\times_\alpha v),$$
$$h(u\times_\alpha  v,u)=h(u\times_\alpha  v,v)=0$$
for all $u,v,w\in P$, $a\in S$.

If the ternary hermitian space $(P,h)$ is orthogonally decomposable then the constructions ${\rm Col}(S,P,h,\alpha)$ is independent of the choice of $\alpha$.

 We obtain split colour algebras as the special case that $S=R\times R$ and $P\cong T\oplus \check{T}$ as $R$-module. In that case
a hermitian form $h$ on $P$ is induced
by the $R$-bilinear form $b:(T\oplus \check{T})\times (T\oplus \check{T})\rightarrow R$, $b((u,\check{u}), (v,\check{v}))=\langle u,\check{v}\rangle +\langle v,\check{u}\rangle$.

\ignore{ We note that $h(u,v)=\frac{1}{2}(h(u,v)+h(v,u))+\frac{1}{2}(h(v,u)-h(u,v))\in S$ with first term on r.h.s being in $R$, the second part being dropped completely in the multiplication of the colour algebra, first slot.}

Denote the noncommutative non-unital 6-dimensional vector colour algebra $(P,\times_\alpha)$ over $R$ that we obtain by projecting the multiplication of ${\rm Col}(S,P,h,\alpha)$ onto $P$ by $W(S,P,h,\alpha)$.

\begin{remark}
(i) We recall that the $R$-module
${\rm Cay}(S,P,h,\alpha)=S\oplus P$ becomes an octonion algebra under the multiplication
$$(a,u)(b,v)=(ab-h(v,u), va+u\bar b+u\times_{\alpha}v)$$
for $u,v\in P$ and $a,b\in S$, with norm
$$n((a,u))=n_S(a)+h(u,u).$$
 Its canonical involution $(a,u)\rightarrow (\overline{a}, -u)$  is a scalar involution. Its norm can also be written as $ n((a,u))=(a,u) \sigma(a,u)= \sigma(a,u)(a, u)$
 and its trace as $ t((a,u))= \sigma(a,u)+(a, u)=(t_{S}(a),0)$.
   \\ (ii) We have the following identities in the octonion algebra $A={\rm Cay}(S,P,h,\alpha)$ \cite{T, Pu2}:
$$h(u\times_\alpha v,u)+\overline{h(u\times_\alpha v,u)}=n_A(u\times_\alpha v,u)=0,$$
$$(u\times_\alpha v) \times_\alpha u=u\times ( v \times_\alpha u),$$
$$h(u, u\times_\alpha v)=0 \text{ and } u\times_\alpha (u\times_\alpha v)=-h(u,u)v+h(v,u)u,$$
$$h( u\times_\alpha v,v)=0  \text{ and }(u\times_\alpha v)\times_\alpha v=h(v,v)u-h(u,v)v$$
 for all $u,v\in P$. Observe that projecting the multiplication of ${\rm Cay}(S,P,h,\alpha)$ onto $P$, we also obtain the algebra $W(S,P,h,\alpha)=(P,\times_\alpha)$.
\\ (iii)
There exist octonion algebras whose norm restricted to their trace zero elements form an indecomposable quadratic space of rank 7 \cite{KPS}, thus have only $R$ as a composition subalgebra, or octonion algebras. A prominent example is Coxeter's order over $\mathbb{Z}$ which does not have any composition subalgebra. This shows that contrary to the situation of base fields, not every octonion algebra over a ring can be used to construct a colour algebra in the usual way.
\end{remark}

\section{Isomorphisms, automorphisms and derivations}\label{sec:iso}

Let $S_i$ be two quadratic \'etale algebras over $R$ with nontrivial automorphism $\can_i$, and
let $(P,h)$ and $(P',h')$ be two nondegenerate $\can_i$-hermitian spaces over $S_i$, $i\in\{1,2\}$. Then  $(P,h)$ and $(P',h')$ are called \emph{isometric}, if there exists an $R$-isomorphism $s:S_1\rightarrow S_2$ and an $s$-semilinear map $\Psi: P\rightarrow P'$ such that
 $h'(\Psi(u),\Psi(v))=s(h(u,v))$ for all $u,v\in P$. We now generalize  \cite[Theorem 2.5]{EM}:

  \begin{proposition}\label{prop:2.5}
  Let $(P,h)$ and $(P',h')$ be two nondegenerate $\can_i$-hermitian spaces over $S_i$ of rank three with trivial determinant, $i\in \{1,2\}$.
  Let $n(v)=h(v,v)$ and $n'(v)=h'(v,v)$ be the associated symmetric bilinear forms on $P$ and $P'$.
  \\
 (i) If $(P,h)$ and $(P',h')$ are isometric  then  $(P,\times_\alpha)=W(S_1,P,h,\alpha)$ and $(P',\times_\alpha')=W(S_2,P',h',\alpha')$ are isomorphic vector colour algebras.
 \\ (ii) If $f:W(S_1,P,h,\alpha)\rightarrow W(S_2,P',h',\alpha')$ is an $F$-algebra isomorphism then $n$ and $n'$ are isometric quadratic forms over $R$.
  \end{proposition}

  \begin{proof}
  (i) Let  $s:S_1\rightarrow S_2$ be an $R$-algebra isomorphism and  $f:P\rightarrow P'$ be an $s$-semilinear map that satisfies
 $h'(f(u),f(v))=s(h(u,v))$ and $\alpha'(f(u),f(v),f(w))=\alpha(u,v,w)$ for all $u,v,w\in P$.
  By assumption, we obtain that
 $h'(f(u),f(v)\times_{\alpha'} f(w))=\alpha'(f(u)\wedge f(v) \wedge f(w)))=s(\alpha(u,v,w))=h(u,v\times_\alpha w)=h'(f(u),f(v) \times_\alpha f(w))$
 for all $u,v,w\in P$.  Since $h'$ is nondegenerate this implies that $f$ is an isomorphism between the vector algebras $(P,\times_\alpha)=W(S_1,P,h,\times_\alpha)$ and $(P',\alpha')=W(S_2,P',h',\alpha')$.
 \\ (ii) Suppose that $f:W(S_1,P,h,\alpha)\rightarrow W(S_2,P',h',\alpha')$ is an algebra isomorphism.
 We know that we have
 $((u\times v)\times v)\times v=n(v) u\times v$ in ${\rm Col}(S_1,P,h,\alpha)$ for all $u,v\in P$ \cite[Lemma 2]{Pu2}
 for all $v\in P$, i.e. $n$ and $n'$ are isometric quadratic forms.
  \end{proof}

Suppose now that $(P,h)$ and $(P',h')$ are two nondegenerate $\can$-hermitian spaces  of rank three  over a quadratic \'etale algebra $S$ with trivial determinant.
Choose two isomorphisms $\alpha: \bigwedge^3 (P,h)\rightarrow \langle 1\rangle $, $\alpha': \bigwedge^3 (P',h')\rightarrow \langle 1\rangle $.
Define two cross products
via $h(u,v\times w)=\alpha(u\wedge v \wedge w)$ and
$h'(u',v'\times w')=\alpha'(u'\wedge v' \wedge w').$

Let $\varphi: P\rightarrow P' $ be an $S$-linear map  such that $\alpha'\circ \det (\varphi)=\alpha$, i.e. $\varphi$ is a (bijective) morphism
 $ (P,\alpha)\rightarrow (P',\alpha') $.
   Assume additionally that $\varphi: (P,\alpha)\rightarrow (P',\alpha') $ is an isometry between $h$ and $h'$.

 \begin{proposition}
 Let $ \varphi:(P,\alpha)\rightarrow (P',\alpha') $
 be a morphism that is an isometry between $h$ and $h'$.
\\ (i)  $\varphi(u\times_{\alpha}v)=u\times_{\alpha'}v$.
 \\ (ii)  $W(S,P,h,\alpha)\cong W(S, P',h',\alpha')$
 via $\varphi$.
\\ (iii)  ${\rm Col}(S, P,h,\alpha)\cong {\rm Col}(S, P',h',\alpha')$
 via $G_\varphi((a,u))=(a,\varphi(u))$.
 \\ (iv)
 ${\rm Cay}(S,P,h,\alpha)\cong {\rm Cay}(S, P',h',\alpha')$
 via $H_\varphi((a,u))=(a,\varphi(u))$.
\end{proposition}

\begin{proof}
(i) We have
$h'(\varphi(u), \varphi(v\times_\alpha w) )=h(u,v\times_\alpha w )= \alpha(u\wedge v\wedge w)=\alpha'(\varphi(u) \wedge \varphi(v)\wedge \varphi(w))=h'(\varphi(u), \varphi(v)\times_{\alpha'} \varphi(w)).$
Since $h$ is nondegenerate it follows that $\varphi(u\times_{\alpha}v)=u\times_{\alpha'}v$.
\\ (ii)  is trivial.
\\ (iii) It is now easy to see that for $G=G_\varphi$, we have
$$G((a,u)(b,v))=G(ab-\frac{1}{2}(h(u,v)+\overline{h(u,v)}), va+u b+u\times_{\alpha}v)$$
$$= (ab-\frac{1}{2}(h(u,v)+\overline{h(u,v)}), \varphi(v)\varphi(a)+\varphi(u) \varphi( b)+\varphi(u\times_{\alpha}v))$$
and
$$G((a,u))G((b,v))$$
$$=(a,\varphi(u))(b,\varphi(v))=(ab-\frac{1}{2}(h'(\varphi(u),\varphi(v))+\overline{h'(\varphi(u),\varphi(v))}), \varphi(v)a+ \varphi(u) b+u\times_{\alpha'} v)$$
$$=(ab-\frac{1}{2}(h(u,v)+\overline{ h(u,v)}),\varphi(v))), \varphi(v)a+ \varphi(u) b+u\times_{\alpha'}v).$$
The second entries are equal due to our assumption that $\varphi$ is $S$-linear.
\\ (iv)
Analogously, we obtain  for $H=H_\varphi$:
$$H((a,u)(b,v))=H(ab-h(u,v), va+u \bar b+u\times_{\alpha}v)= (ab-h(u,v), \varphi(v)\varphi(a)+\varphi(u) \varphi( \bar b)+\varphi(u\times_{\alpha}v))$$
and
$$H((a,u))H((b,v))=(a,\varphi(u))(b,\varphi(v))=(ab-h'(\varphi(u),\varphi(v)), \varphi(v)a+ \varphi(u) \bar b+u\times_{\alpha'} v)$$
$$=(ab- h(u,v),\varphi(v))), \varphi(v)a+ \varphi(u) \bar b+u\times_{\alpha'}v).$$
\end{proof}

\begin{corollary}\label{cor:4.3}
If  $ \varphi:(P,\alpha)\rightarrow (P',\alpha') $
 is a morphism that is an isometry between $h$ and $h'$, such that ${\rm det}(\varphi)=id$, then
$G_\varphi \in {\rm Aut}({\rm Col}_R(S, P, h,\alpha))$ and $H_\varphi \in{\rm Aut}_R({\rm Cay}(S, P, h,\alpha))$.
\end{corollary}

\begin{theorem}
The maps $W(S, P, h,\alpha)\mapsto {\rm Cay}(S, P, h,\alpha)$ and $W(S, P, h,\alpha)\mapsto {\rm Cay}(S, P, h,\alpha)$ induce bijections between the following sets of algebras:
\\ (i) The set of isomorphism classes of pairs $({\rm Cay}(S, P, h,\alpha), S)$ of octonion algebras over $R$ with subalgebra $S$.
\\ (ii)  The set of isomorphism classes of colour algebras $({\rm Col}(S, P, h,\alpha), S)$ over $R$.
\\ (iii)  The set of isomorphism classes  of vector algebras $W(S, P, h,\alpha)$ over $R$.
\end{theorem}

This generalizes \cite[Theorem 3.1]{EM}.

\begin{proof}
Given a vector algebra $W(S, P, h,\alpha)$ and $S$ we can define $({\rm Cay}(S, P, h,\alpha)), S)$ and $({\rm Col}(S, P, h,\alpha)), S)$.
Conversely, the multiplication of  $({\rm Cay}(S, P, h,\alpha), S)$ and $({\rm Col}(S, P, h,\alpha), S)$ restricted to $P$ yields the algebra
$W(S, P, h,\alpha)$ over $R$.
\end{proof}

Let $C={\rm Col}(S, P,h,\alpha)$, $A={\rm Cay}(S, P,h,\alpha)$ and
$${\rm Aut}_S(A)=\{ f\in {\rm Aut}(A), f|_S=id \},\quad {\rm Der}_S(A)=\{  d\in {\rm Der}(A), d|_S=0\}.$$
Because of the way the multiplications on $A$ and $C$ are defined, it is straightforward to see that  ${\rm Der}(W(S, P,h,\alpha))$ can be embedded into  ${\rm Der}(A)$ and  ${\rm Der}(C)$. Since $d(1)=0$ for every $d\in {\rm Der}(C)$, we have $d|_P\in {\rm Der}(W(S, P,h,\alpha))$. Since
 $d|_S=0$ for every $d\in {\rm Der}_S(A)$, we have $d|_P\in {\rm Der}(W(S, P,h,\alpha))$.
 Thus as over base fields, we obtain:

 \begin{lemma}
 ${\rm Der}({\rm Col}(S, P,h,\alpha))={\rm Der}({\rm Cay}_S(S, P,h,\alpha))={\rm Der}(W(S, P,h,\alpha)).$
 \end{lemma}

 Define the Lie algebras
 $$u(P,h)=\{f\in {\rm End}_S(P)\,|\, h(f(u),v)+h(u,f(v))=0 \text{ for all }u,v\in P\},$$
 $$su(P,h)=\{f\in u(P,h)\,|\, tr_S(f)=0\}.$$
  When $R=F$
  is a field, ${\rm Der}(W(S, P,h,\alpha))=su(P,h)$, when $S=F\times F$ we have hence
 ${\rm Der}(W(S, P,h,\alpha))=sl(3,F)$ \cite[Theorem 3.4]{EM}.

 For $u,w\in P$ define the $S$-linear map $\lambda_{u,v}:P\rightarrow P,$ $x\mapsto h(x,u)v-h(x,v)u.$  Then the $R$-linear span of the maps $\lambda_{u,v}$ with $u,v$ running through all elements in $P$ is contained in $u(P,h)$. When $R$ is a field then $u(P,h)$ equals this $R$-linear span of the maps $\lambda_{u,v}$.

\begin{proposition}
  ${\rm Aut}_S(C)={\rm Aut}(W(S, P,h,\alpha))\cap {\rm End}_S{P}$ and ${\rm Aut}_R(C)={\rm Aut}(W(S, P,h,\alpha))$.
\end{proposition}

\begin{proof}
Similarly as observed in \cite[p.1301]{EM} for base fields, ${\rm Aut}(W(S, P,h,\alpha))$ can be embedded into  ${\rm Aut}(C)$ via $f\mapsto G_f$ with $G_f((a,u))=(a,f(u))$ for all $a\in R$, $u\in P$.
Put $H={\rm Aut}(W(S, P,h,\alpha))\cap {\rm End}_S{P}$. Then every $f\in H$ can be extended to $G_f\in {\rm Aut}_S(C)$ and we can embed $H$ into ${\rm Aut}_S(C)$ via $f\mapsto G_f$.
Conversely, because of the multiplicative structure of $C$, every $G\in {\rm Aut}_S(C)$ restricts to an automorphism $G|_P$ in ${\rm Aut}(W(S, P,h,\alpha))$ and an endomorphism in ${\rm End}_S{P}$, and every $G\in {\rm Aut}(C)$ restricts to an automorphism $G|_P$ in ${\rm Aut}(W(S, P,h,\alpha))$.
This implies that  ${\rm Aut}_S(C)=H$  and that ${\rm Aut}(C)={\rm Aut}(W(S, P,h,\alpha))$.
\end{proof}

Define the unitary group with respect to $h$,
$$U(P,h)=\{ \varphi\in {\rm GL}(P,S)\,|\, h(\varphi(u),\varphi(v))=h(u,v) \text{ for all }u,v\in P\},$$
as the set of isometries of $h$ and the special unitary group with respect to $h$ as
$$SU(P,h)=\{ \varphi\in U(P,h)\,|\, \det (\varphi)=1\}.$$
 By Corollary \ref{cor:4.3},
 $$SU(P,h)\subset {\rm Aut}({\rm Col}(S, P, h,\alpha)) \text{ and } SU(P,h)\subset{\rm Aut}({\rm Cay}(S, P, h,\alpha)).$$
When $R$ is a field then ${\rm Aut}_S(C)=SU(P,h)$.

\section{Natural examples of noncommutative Jordan algebras with big radicals} \label{sec:Jordan}

  The theory of  composition algebras over schemes was  launched by Petersson \cite{P2} and later extended to large classes of Jordan algebras, e.g. \cite{Ach1, Ach2, PST, PST2, Pu1}.
  We refer the reader to \cite{P2} or later work for the details on how to transfer the language of nonassociative algebras over base rings to algebras over base schemes (and vice versa). Our goal is to point out how much richer the theory of  colour algebras becomes when we study them over schemes, adjusting the approach of \cite[3.8]{P2} in the following. To avoid pathological cases, we only look at algebras over schemes that are defined over locally free $\mathcal{O}_X$-modules of constant rank 7 and have full support in the sense of \cite{P2}. A \emph{colour algebra} over $X$ is then defined as a unital algebra over $X$ that is a colour algebra for all $P\in X$.

Let $S=R[t_0,\dots,t_n]$ be  equipped with the well-known canonical grading $S=\oplus_{m\geq 0}S_m$, where ${\rm rank}\, S_m=\binom{m+n}{n}$. Let $X={\rm Proj}\,S=\mathbb{P}_R^n$.
For all $m\in \mathbb{Z}$, $\mathcal{O}_X(m)$ is a locally free $\mathcal{O}_X$-module of rank one, and the global sections satisfy
$$H^0(X,\mathcal{O}_X(m))=S_m \text{ for }m\geq 0,\quad H^0(X,\mathcal{O}_X(m))=0 \text{ for }m< 0.$$
For all $m,l\in \mathbb{Z}$, there exists an isomorphism
$$\alpha:{\rm det}(\mathcal{O}_X(l)\oplus \mathcal{O}_X(m)\oplus \mathcal{O}_X(-l-m))\rightarrow \mathcal{O}_X.$$
For all positive integers  $l,m$, we define  the split colour algebra
$$\mathcal{C}_{l,m}={\rm Col}(\mathcal{O}_X\times \mathcal{O}_X, \mathcal{O}_X(l)\oplus \mathcal{O}_X(m)
\oplus \mathcal{O}_X(-l-m),\alpha),$$
$$
\left [\begin {array}{cc}
 a & u\\
\check{u}  &  a \\
\end {array}\right ]
\left [\begin {array}{cc}
 b & v\\
\check{v}  &  b \\
\end {array}\right ]
=
\left [\begin {array}{cc}
  ab+\langle u,\check{v}\rangle + \langle v,\check{u}\rangle & av+bu-\check{u}\times \check{v}\\
b\check{u} +a\check{v}+u\times v &  ab+\langle u,\check{v}\rangle + \langle v,\check{u}\rangle \\
\end {array}\right ].
$$
 over $X$. The colour algebra $\mathcal{C}_{l,m}$ is a quadratic algebra  with norm
 $$
n(
\left [\begin {array}{cc}
 a & u\\
\check{u}  &  a \\
\end {array}\right ])=
a^2-\langle u,\check{u}\rangle.$$
The localizations of the algebras $\mathcal{C}_{l,m}$ at $P\in X$ are split colour algebras over $\mathcal{O}_{P,X}$, and the algebras ${(\mathcal{C}_{l,m})}_P\otimes_{\mathcal{O}_{P,X}} k(P)\cong {\rm Zor}(k(P))$ are split colour algebras over the residue class fields $k(P)$, for all $P\in X$.

 The global sections $H^0(X, \mathcal{C}_{l,m})$ of a split colour algebra
$\mathcal{C}_{l,m}$ provide us with canonical examples of flexible quadratic algebras $ C_{l,m}=H^0(X, \mathcal{C}_{l,m})$ over $R=H^0(X, \mathcal{O}_X)$: the algebra
\[ C_{l,m}=
\left [\begin {array}{cc}
R&S_l\oplus S_m\\
S_{l+m}&R\\
\end {array}\right ]
\]
is defined on the free module $R^s$ with
$$s=1+\binom{l+n}{n}+\binom{m+n}{n}+\binom{(l+m)+n}{n}$$
 and is an $R$-subalgebra  of $ {\rm Col}(S)$, hence is a noncommutative Jordan $R$-subalgebra of $ {\rm Col}(S)$. For instance, if $n=1$ then $s=
5+2(l+m)$
is always odd. Its multiplication is defined via
\[
\left [\begin {array}{cc}
 a & f_l\oplus f_m\\
f_{l+m}  &  a \\
\end {array}\right ]
\left [\begin {array}{cc}
 b & g_l\oplus g_m\\
g_{l+m}  &  b \\
\end {array}\right ]
=
\left [\begin {array}{cc}
 ab & (ag_l+bf_l)\oplus (ag_m+bf_m)\\
bf_{l+m}+ag_{l+m}+f_lg_m-f_mg_l &  ab \\
\end {array}\right ]
\]
for $a,b\in R$, where the $f$'s and $g$'s are homogeneous polynomials in $S$ with subscripts
indicating their degrees.
Note that here the terms corresponding to $\langle u,\check{v}\rangle $ and $\langle v,\check{u}\rangle$ vanish, $u\times v$ corresponds to $(f_l\oplus f_m)\otimes ( g_l\oplus g_m)=f_l g_m-f_m g_l$
and $\check{u}\times \check{v}$ corresponds to $f_{l+m} \times g_{l+m} =0$ for all elements $f,g$.

Let $n$ be the norm of $\mathcal{C}_{l,m}$, and let $n_0=n(X):H^0(X, \mathcal{C}_{l,m})\to H^0(X, \mathcal{O}_X)$.
 If $R$ is a field then
 $$
n_0(
\left [\begin {array}{cc}
 a & u\\
\check{u}  &  a \\
\end {array}\right ])=
a^2$$
is  degenerate, and  $C_{l,m}$ has the radical
\[{\rm rad}\,C_{l,m}= {\rm rad}\,n_0=
\left [\begin {array}{cc}
0&S_l\oplus S_m\\
\noalign{\smallskip}
S_{l+m}&0\\
\end {array}\right ]
\]
 of dimension
$\binom{l+n}{n}+\binom{m+n}{n}+\binom{(l+m)+n}{n}$,
analogously as observed in \cite[3.8]{P2}.

\end{document}